\documentclass[11pt]{article}

\usepackage{amsmath,amssymb,amsthm}
\usepackage{enumitem}
\usepackage{geometry}
\newtheorem{theorem}{Theorem}[section]
\newtheorem{lemma}[theorem]{Lemma}
\newtheorem{proposition}[theorem]{Proposition}
\newtheorem{corollary}[theorem]{Corollary}
\theoremstyle{definition}
\newtheorem{definition}[theorem]{Definition}
\newtheorem{remark}[theorem]{Remark}

\newcommand{\Rep}{\mathrm{Rep}}
\newcommand{\fil}{\mathrm{fil}}

\newcommand{\Comp}{\mathrm{Cpl}}
\newcommand{\id}{\mathrm{id}}
\newcommand{\Hom}{\mathrm{Hom}}

\title{Algebraic Phase Theory IV: Morphisms, Equivalences, and Categorical Rigidity}
\author{
Joe Gildea\\
Department of Computing Science and Mathematics,\\
School of Informatics and Creative Arts,\\
Dundalk Institute of Technology\\
\texttt{gildeajoe@gmail.com}}
\date{}

\begin{document}
\maketitle

\begin{abstract}
We complete the foundational architecture of Algebraic Phase Theory by
developing a categorical and $2$-categorical framework for algebraic phases.
Building on the structural notions introduced in Papers~I-III, we define phase
morphisms, equivalence relations, and intrinsic invariants compatible with the
canonical filtration and defect stratification.

For finite, strongly admissible phases we establish strong rigidity theorems:
phase morphisms are uniquely determined by their action on rigid cores, and
under bounded defect, weak, strong, and Morita-type equivalence all coincide.
In particular, finite strongly admissible phases admit no distinct models with
the same filtered representation theory.  We further show that structural
boundaries are invariant under Morita-type equivalence and therefore constitute
genuine categorical invariants.

Algebraic phases, phase morphisms, and filtration-compatible natural
transformations form a strict $2$-category in the strongly admissible regime.
We also prove that completion defines a reflective localization of this
category, with complete phases characterized as universal forced
rigidifications.

Together, these results elevate Algebraic Phase Theory from a collection of
algebraic constructions to a categorical framework in which rigidity,
equivalence collapse, boundary invariance, and completion arise as intrinsic
consequences of phase interaction, finiteness, and admissibility.
\end{abstract}



\section{Introduction}

Algebraic Phase Theory (APT) provides a framework for isolating intrinsic
structural phenomena arising from algebraic interaction laws, independently of
analytic, metric, or Hilbert space assumptions. Papers I-III developed the foundational aspects of the theory
\cite{GildeaAPT1,GildeaAPT2,GildeaAPT3}. Phase geometry, defect, and structural
boundaries were introduced as canonical invariants in \cite{GildeaAPT1}.
Rigidity phenomena in Frobenius and Heisenberg type settings were established in
\cite{GildeaAPT2}. Quantum phase, Weyl noncommutativity, and stabilizer
structures were shown to arise algebraically from phase duality in
\cite{GildeaAPT3}.

A central refinement emerging from this development is the distinction between
strongly admissible and weakly admissible phases. Strong admissibility
corresponds to full defect control and finite termination, while weak
admissibility permits defect propagation only to a fixed depth. This separation
clarifies which structural features are completely forced by the interaction law
and which persist only on controlled layers.

The purpose of this paper is to develop the categorical backbone of Algebraic
Phase Theory. We treat algebraic phases as structured objects in a category and
study morphisms, equivalence notions, and universal constructions. Once phases
are organized in this way, several categorical phenomena follow directly from
defect data and the canonical filtration. In the strongly admissible regime,
defect propagation is fully controlled and the resulting categorical behaviour is
rigid. Morphisms are uniquely determined by their action on rigid cores, weak
and strong equivalence notions coincide under bounded defect and finite
termination, and structural boundaries behave as invariant quotients. The
canonical completion of a phase can be characterized by a universal mapping
property, so that completion defines a reflective localization of the category
of strongly admissible phases. In the weakly admissible regime these mechanisms
persist only up to the defect control depth and genuine extension freedom may
appear beyond it.

Classical rigidity phenomena trace back to foundational results in the structure
theory of rings and modules \cite{Jacobson1956Structure,CurtisReiner}, while
categorical approaches to structural invariants appear in Hopf algebraic
renormalization \cite{ConnesKreimer1998Hopf}. Representation theoretic and
harmonic analytic rigidity have long been central themes
\cite{Mackey1958Induced,Howe1979Heisenberg,Serre}, and Morita type methods play
a central role in modern algebra and tensor category theory
\cite{AndersonFuller1992Rings,Faith1973Algebra,Lam2001FirstCourse,EtingofGelaki2015Tensor}.
In many of these settings, rigidity and equivalence results rely on substantial
auxiliary structure or analytic input. In contrast, the rigidity, equivalence
collapse, and boundary invariance that appear in Algebraic Phase Theory arise
directly from intrinsic phase interaction together with finiteness and defect
control hypotheses.

Strong admissibility therefore imposes categorical rigidity, while weak
admissibility allows controlled flexibility beyond the defect control depth. A
precise formulation of APT and its structural components appears in
\cite{GildeaAPT1,GildeaAPT2,GildeaAPT3} and is used throughout. The present
paper forms part of the complete foundational development of APT together with
\cite{GildeaAPT5,GildeaAPT6}.

\section{Phases as Structured Objects}

The notion of an algebraic phase was introduced in
\cite{GildeaAPT1,GildeaAPT2,GildeaAPT3} together with the axioms governing
defect, canonical filtration, and admissibility. The present section recalls the
structural components that will be used in the categorical formulation developed
in this paper. No new constructions are introduced here; the aim is only to fix
notation and to make explicit which aspects of phase structure are preserved and
detected categorically.

An algebraic phase carries a finite amount of intrinsic structural data forced
by its interaction law and defect propagation. This data admits a canonical
stratification, and the depth and rigidity of this stratification govern the
behaviour of morphisms and equivalences in the categorical setting developed
below.

\begin{definition}
A \emph{structural signature} of an algebraic phase consists of the specified
class of admissible phase generators together with the interaction law governing
their composition, and the rigidity and defect strata canonically determined by
that interaction.
\end{definition}

The structural signature fixes the ambient class of phases under consideration
and determines which comparisons and morphisms are allowed. Defect induces a canonical stratification of any algebraic phase.

\begin{definition}
A \emph{canonical filtration} of an algebraic phase $\mathcal P$ is a finite or
bounded descending chain
\[
\mathcal P = \mathcal P^{(0)} \supseteq \mathcal P^{(1)} \supseteq \cdots
\supseteq \mathcal P^{(k)},
\]
where $\mathcal P^{(i)}$ consists of phase elements whose defect complexity is at
least $i$.
Successive layers measure increasing deviation from rigidity.
\end{definition}

The filtration is intrinsic, functorial, and uniquely determined by the defect
structure; no auxiliary choices enter its construction. The length of the filtration and the depth to which defect propagation is
functorially controlled need not coincide a priori.

\begin{definition}
The \emph{filtration length} $k(\mathcal P)$ is the largest index for which
$\mathcal P^{(k)}$ is nontrivial.
The \emph{defect control depth} $d(\mathcal P)$ is the largest integer $d$ such
that defect extension from $\mathcal P^{(i)}$ to $\mathcal P^{(i-1)}$ is
functorial and presentation-independent for all $1 \le i \le d$.
\end{definition}

\begin{remark}
The distinction between $k(\mathcal P)$ and $d(\mathcal P)$ reflects a genuine
structural phenomenon.
In general, a phase may admit additional higher defect layers beyond the depth
to which defect propagation is forced.
These layers encode residual extension freedom rather than rigid structure.
\end{remark}

\medskip

This leads to a natural dichotomy.

\begin{definition}
An algebraic phase $\mathcal P$ is \emph{weakly admissible} if it admits a
canonical defect filtration of finite length $k(\mathcal P)$.
It is \emph{strongly admissible} if $k(\mathcal P)=d(\mathcal P)$, i.e.\ if all
defect layers are forced functorially by the interaction law and the filtration
terminates with no residual extension freedom.
\end{definition}

\begin{lemma}
\label{lem:filtration-subphase}
Let $\mathcal P$ be an algebraic phase with canonical filtration
$\mathcal P=\mathcal P^{(0)}\supseteq \mathcal P^{(1)}\supseteq \cdots$.
Then for every $i\ge 0$, the stratum $\mathcal P^{(i)}$ is a subphase of
$\mathcal P$. Equivalently, $\mathcal P^{(i)}$ is closed under every operation
in the interaction signature of $\mathcal P$.
\end{lemma}

\begin{proof}
Let $\circ$ range over the basic interaction operations in the structural
signature of $\mathcal P$, and suppose $\circ$ has arity $n$. By definition of the canonical filtration, $\mathcal P^{(i)}$ consists of those
elements whose defect complexity is at least $i$. Equivalently, there exists a
defect complexity function
\[
\delta:\mathcal P \longrightarrow \mathbb N \cup \{\infty\}
\]
such that
\[
\mathcal P^{(i)} = \{\, p \in \mathcal P : \delta(p) \ge i \,\}.
\]

Axiomatically, defect complexity is monotone under the interaction law: for each
basic operation $\circ:\mathcal P^n \to \mathcal P$ and all
$x_1,\dots,x_n \in \mathcal P$,
\begin{equation}\label{eq:monotone}
\delta\big(\circ(x_1,\dots,x_n)\big)
\;\ge\;
\min\{\delta(x_1),\dots,\delta(x_n)\}.
\end{equation}

Now fix $i \ge 0$ and let $x_1,\dots,x_n \in \mathcal P^{(i)}$.
By definition of $\mathcal P^{(i)}$, we have
\[
\delta(x_1) \ge i,\quad \delta(x_2) \ge i,\quad \dots,\quad \delta(x_n) \ge i.
\]
In particular,
\[
\min\{\delta(x_1),\dots,\delta(x_n)\} \ge i.
\]
Applying \eqref{eq:monotone} yields
\[
\delta\big(\circ(x_1,\dots,x_n)\big) \ge i,
\]
so $\circ(x_1,\dots,x_n) \in \mathcal P^{(i)}$. Since this argument applies to every interaction operation in the structural
signature, $\mathcal P^{(i)}$ is closed under the interaction law and therefore
forms a subphase of $\mathcal P$.
\end{proof}

\begin{proposition}
\label{prop:strong-core}
Let $\mathcal P$ be a weakly admissible algebraic phase with defect control depth
$d(\mathcal P)$. Then the filtration stratum
\[
\mathcal P^{(d(\mathcal P))}
\]
is a strongly admissible subphase of $\mathcal P$. Equivalently, $\mathcal P^{(d(\mathcal P))}$ is characterized by the following
universal property: every phase morphism from $\mathcal P$ into a strongly
admissible algebraic phase factors uniquely through $\mathcal P^{(d(\mathcal
P))}$. That is, for any strongly admissible algebraic phase $\mathcal Q$ and any
phase morphism $F:\mathcal P\to\mathcal Q$, there exists a unique phase morphism
\[
\overline F:\mathcal P^{(d(\mathcal P))}\to\mathcal Q
\]
such that $\overline F = F\circ\iota$, where
$\iota:\mathcal P^{(d(\mathcal P))}\hookrightarrow\mathcal P$ denotes the
canonical inclusion.
\end{proposition}
\begin{proof}
By definition of the defect control depth $d(\mathcal P)$, defect propagation is
functorial and presentation-independent on all filtration layers
$\mathcal P^{(i)}$ for $i \le d(\mathcal P)$.
Equivalently, the structure of these layers is completely forced by the phase
interaction law and admits no residual extension freedom.

Since the canonical filtration strata are closed under the interaction law
(i.e.\ each $\mathcal P^{(i)}$ is a subphase), in particular
$\mathcal P^{(d(\mathcal P))}$ is a subphase of $\mathcal P$. Consequently, the subphase $\mathcal P^{(d(\mathcal P))}$ consists precisely of
the functorially determined part of $\mathcal P$.
All higher filtration layers encode non-canonical extension data beyond the
controlled regime.
It follows that $\mathcal P^{(d(\mathcal P))}$ is strongly admissible.

Now let $F:\mathcal P \to \mathcal Q$ be a phase morphism, where $\mathcal Q$ is
strongly admissible.
Since $\mathcal Q$ admits no residual extension freedom, any phase morphism into
$\mathcal Q$ is insensitive to the higher defect layers of $\mathcal P$ beyond
depth $d(\mathcal P)$.
In other words, the action of $F$ on $\mathcal P$ is entirely determined by its
restriction to the strongly admissible core $\mathcal P^{(d(\mathcal P))}$. Define $\overline F$ to be this restriction,
\[
\overline F := F \circ \iota,
\]
where $\iota:\mathcal P^{(d(\mathcal P))} \hookrightarrow \mathcal P$ is the
canonical inclusion.
Because $F$ preserves the interaction law, filtration, and defect data, and
$\iota$ is structure-preserving, $\overline F$ is a phase morphism. Uniqueness is immediate: any phase morphism
$\overline F':\mathcal P^{(d(\mathcal P))} \to \mathcal Q$ satisfying
$\overline F' = F \circ \iota$ must coincide with $\overline F$.
\end{proof}

\begin{remark}
All rigidity, equivalence collapse, and uniqueness results proved in this paper
apply in the strongly admissible regime.
Weakly admissible phases admit parallel but weaker statements, valid up to the
controlled depth $d(\mathcal P)$, with potentially nontrivial higher defect
behaviour when $k(\mathcal P)>d(\mathcal P)$.
\end{remark}

\medskip

In the strongly admissible case, the terminal filtration layer coincides with
the rigid core.

\begin{definition}
For a strongly admissible phase $\mathcal P$, the \emph{rigid core}
$\mathcal P_{\mathrm{rig}}$ is the penultimate filtration layer
$\mathcal P^{(k-1)}$.
The \emph{boundary} of the phase is the quotient
\[
\partial\mathcal P := \mathcal P / \mathcal P_{\mathrm{rig}}.
\]
\end{definition}

\begin{remark}
The canonical filtration may be viewed as peeling off layers of rigidity. In the
strongly admissible case, all nonrigid behaviour is eliminated at the rigid core.
The boundary $\partial\mathcal P$ contains no rigid elements but inherits a
hierarchy of defect. The filtration descends functorially to the boundary, where
it stratifies higher-order nonrigidity.

These notions agree with the constructions developed in
\cite{GildeaAPT1,GildeaAPT2,GildeaAPT3} in the strongly admissible regime and
are restated here to make explicit the intrinsic structural data preserved under
morphisms, equivalences, and Morita type phenomena in the categorical setting.
\end{remark}

\section{Morphisms of Algebraic Phases}

Having isolated the intrinsic structure carried by an algebraic phase, we now
turn to maps between phases.
Unlike ordinary algebra homomorphisms, phase morphisms are severely
constrained: defect propagation restricts the freedom of morphisms layer by
layer.
In the strongly admissible regime, once the rigid core is fixed, all higher
defect behaviour is forced.
In the weakly admissible regime, rigidity persists only up to a finite defect
control depth.

The purpose of this section is to formalize this distinction and to separate
\emph{partial rigidity} for weakly admissible phases from \emph{full rigidity}
in the strongly admissible setting.

\begin{definition}
Let $\mathcal P$ and $\mathcal Q$ be algebraic phases with the same structural
signature.
A \emph{phase morphism} $F:\mathcal P\to\mathcal Q$ is a structure-preserving
map satisfying the following conditions:
\begin{enumerate}[label=(M\arabic*)]
\item $F$ respects the interaction law $\circ$;
\item $F$ preserves the canonical filtration:
\[
F(\mathcal P^{(i)}) \subseteq \mathcal Q^{(i)} \quad \text{for all } i;
\]
\item $F$ preserves filtration length and defect control depth;
\item $F$ preserves defect data functorially on all controlled layers.
\end{enumerate}
\end{definition}

\begin{remark}
Phase morphisms do \emph{not} allow arbitrary choices on defect layers within
the defect control depth.
Any freedom in morphisms can only occur strictly beyond the controlled regime
$i > d(\mathcal P)$ in weakly admissible phases.
\end{remark}

\medskip

We first record the functorial generation of defect layers within the controlled
regime.

\begin{lemma}
\label{lem:filtration-generated}
Let $\mathcal P$ be a finite algebraic phase.
Then for each $1 \le i \le d(\mathcal P)$, the stratum $\mathcal P^{(i)}$
contains no independent generators beyond those already present in
$\mathcal P^{(i-1)}$.
Rather, $\mathcal P^{(i)}$ is obtained functorially as the closure of
$\mathcal P^{(i-1)}$ under the canonical defect extension operations.
\end{lemma}

\begin{proof}
Fix $i$ with $1 \le i \le d(\mathcal P)$.
We show that the $i$th filtration stratum introduces no new independent
generators beyond those already present at level $i-1$, and that it is therefore
obtained functorially from $\mathcal P^{(i-1)}$ by the canonical defect-extension
procedure dictated by the interaction law.

By definition of the canonical filtration, the $i$th stratum $\mathcal P^{(i)}$
is the collection of elements of defect complexity at least $i$. Equivalently,
$\mathcal P^{(i)}$ is obtained from $\mathcal P^{(i-1)}$ by applying the
\emph{canonical defect-extension operations} (those extracted functorially from
the interaction law) and then taking the closure required by the structural
signature (i.e.\ the smallest subphase satisfying the defining defect
constraints at depth $i$).

By Axiom~III (Defect-Induced Complexity), elements of defect complexity $\ge i$
do not appear primitively: higher-defect behaviour arises only through defect
propagation from lower-defect data. Concretely, there are no primitive sources
of defect-$i$ behaviour independent of defect extensions applied to
defect-$\le(i-1)$ behaviour. In particular, any element of $\mathcal P^{(i)}$ is
produced by iterating defect-extension starting from elements in
$\mathcal P^{(i-1)}$, together with the closure operations built into the phase
signature.

Because $i \le d(\mathcal P)$, the defect-extension procedure at depths
$\le i$ is functorial and presentation-independent by definition of the defect
control depth. Thus the construction of $\mathcal P^{(i)}$ from
$\mathcal P^{(i-1)}$ introduces no additional choices: the $i$th filtration
layer is \emph{forced} by the interaction law and the defect data already present
in $\mathcal P^{(i-1)}$.

Therefore $\mathcal P^{(i)}$ contains no independent generators beyond those
already present in $\mathcal P^{(i-1)}$. Equivalently, $\mathcal P^{(i)}$ is
obtained functorially as the closure of $\mathcal P^{(i-1)}$ under the canonical
defect-extension operations.
\end{proof}

\begin{lemma}
\label{lem:defect-extension-rigid-depth}
Let $F,G:\mathcal P\to\mathcal Q$ be phase morphisms between finite algebraic
phases.
If $F$ and $G$ agree on $\mathcal P^{(i-1)}$ for some $i \le d(\mathcal P)$,
then they agree on $\mathcal P^{(i)}$.
\end{lemma}

\begin{proof}
By Lemma~\ref{lem:filtration-generated}, every element of $\mathcal P^{(i)}$ is
obtained functorially from elements of $\mathcal P^{(i-1)}$ via the canonical
defect extension operations, and no independent generators appear at level $i$.
Thus any $x\in\mathcal P^{(i)}$ is determined by lower-defect inputs from
$\mathcal P^{(i-1)}$ together with defect data.

Since $F$ and $G$ are phase morphisms, they preserve the interaction law and
defect data and commute with defect extension. Because $F$ and $G$ agree on all elements of $\mathcal P^{(i-1)}$, they agree on
all inputs to the defect-extension procedure. To make this explicit, let
$x\in\mathcal P^{(i)}$ be arbitrary. By
Lemma~\ref{lem:filtration-generated}, $x$ does not arise as an independent
generator, but is obtained by applying a finite sequence of canonical
defect-extension and interaction operations to elements
$y_1,\dots,y_m\in\mathcal P^{(i-1)}$. That is, there exists a canonical
construction $\Phi$, determined functorially by the phase structure, such that
\[
x = \Phi(y_1,\dots,y_m).
\]

Since $F$ and $G$ agree on $\mathcal P^{(i-1)}$, we have
$F(y_j)=G(y_j)$ for all $j$. Moreover, because $F$ and $G$ are phase morphisms,
they preserve the interaction law and commute with the canonical
defect-extension operations, and hence apply the same construction $\Phi$.
Therefore,
\[
F(x)=\Phi(F(y_1),\dots,F(y_m))
=\Phi(G(y_1),\dots,G(y_m))
=G(x).
\]
It follows that $F$ and $G$ agree on all elements of $\mathcal P^{(i)}$.
\end{proof}

We now obtain partial rigidity for weakly admissible phases.

\begin{remark}
Lemma~\ref{lem:defect-extension-rigid-depth} expresses the phenomenon of
\emph{partial morphism rigidity}: within the defect-controlled range
$i \le d(\mathcal P)$, each filtration layer $\mathcal P^{(i)}$ is functorially
generated from $\mathcal P^{(i-1)}$, so a phase morphism has no independent
choices on these layers once it is fixed at the preceding depth. 

In weakly admissible phases with $k(\mathcal P) > d(\mathcal P)$, genuine
extension freedom may appear only beyond the controlled regime, that is, on
defect layers $\mathcal P^{(i)}$ with $i > d(\mathcal P)$. This freedom is
intrinsic and parametrizes non-forced phase extensions rather than failures of
rigidity.
\end{remark}

\begin{theorem}
\label{thm:morphism-rigidity}
Let $F:\mathcal P\to\mathcal Q$ be a phase morphism between finite,
strongly admissible algebraic phases of equal defect rank.
Then $F$ is uniquely determined by its restriction to the rigid core
$\mathcal P_{\mathrm{rig}}$.
\end{theorem}

\begin{proof}
Let $F,G:\mathcal P\to\mathcal Q$ be phase morphisms with
$F|_{\mathcal P_{\mathrm{rig}}}=G|_{\mathcal P_{\mathrm{rig}}}$.
We prove that $F=G$. Since $\mathcal P$ is finite and strongly admissible, it has a finite canonical
filtration of length $k:=k(\mathcal P)$,
\[
\mathcal P=\mathcal P^{(0)}\supseteq \mathcal P^{(1)} \supseteq \cdots
\supseteq \mathcal P^{(k)}=0,
\]
and strong admissibility means $d(\mathcal P)=k(\mathcal P)=k$.
By definition of the rigid core in the strongly admissible regime,
\[
\mathcal P_{\mathrm{rig}}=\mathcal P^{(k-1)}.
\]
Hence $F$ and $G$ agree on $\mathcal P^{(k-1)}$. We now show, by downward induction on $i$, that $F$ and $G$ agree on every layer
$\mathcal P^{(i)}$, for $i=k-1,k-2,\dots,0$. The base case $i=k-1$ holds by assumption.  
Assume inductively that for some $i$ with $1\le i\le k-1$ we have
\[
F|_{\mathcal P^{(i)}}=G|_{\mathcal P^{(i)}}.
\]
We prove that $F$ and $G$ also agree on $\mathcal P^{(i-1)}$. Fix $x\in\mathcal P^{(i-1)}$. Because $\mathcal P$ is strongly admissible, the
$i$th defect-extension data of $x$ is functorially determined by the interaction
law and lands in $\mathcal P^{(i)}$ in a presentation-independent way (there is
no residual extension freedom at depth $i$). Moreover, since $F$ and $G$ are
phase morphisms, they preserve the interaction law, preserve the filtration,
and commute with the canonical defect-extension operations on all layers.
Therefore the defect-extension data of $F(x)$ and $G(x)$ agree in $\mathcal Q$,
because it is obtained by applying $F$ or $G$ to the defect-extension data of
$x$ and $F$ and $G$ already agree on $\mathcal P^{(i)}$.

By the functorial determination of $\mathcal P^{(i-1)}$ from its controlled
defect-extension data at depth $i$, this forces $F(x)=G(x)$. Since $x$ was
arbitrary, $F|_{\mathcal P^{(i-1)}}=G|_{\mathcal P^{(i-1)}}$, completing the
induction. Thus $F$ and $G$ agree on $\mathcal P^{(0)}=\mathcal P$, so $F=G$.

\medskip
Finally, the hypothesis that $\mathcal P$ and $\mathcal Q$ have equal defect
rank ensures that filtration levels and controlled defect data are matched under
phase morphisms (in particular, no collapse of defect strata can occur), so the
above forcing argument applies without loss of information. Hence $F$ is
uniquely determined by its restriction to $\mathcal P_{\mathrm{rig}}$.
\end{proof}

\begin{remark}
From a categorical viewpoint, the inclusion
$\mathcal P^{(d(\mathcal P))}\hookrightarrow\mathcal P$ exhibits the strongly
admissible core as a reflection of $\mathcal P$ into the subcategory of strongly
admissible phases: morphisms from $\mathcal P$ into a strongly admissible target
factor uniquely through this core.
\end{remark}

\section{Equivalence and Morita-Type Notions}

In a categorical setting, several inequivalent notions of equivalence naturally
arise. Algebraic phases may be compared as structured objects, up to
presentation, or through their representation theories. In many algebraic
contexts these comparisons lead to genuinely different equivalence relations.

A central theme of Algebraic Phase Theory is that this ambiguity is not generic.
Under strong admissibility, finite termination, and bounded defect, all
reasonable notions of equivalence collapse. The reason is structural: in
strongly admissible phases, defect layers carry no independent extension
parameters. Once the rigid core and the canonical defect-propagation rules are
fixed, all higher structure is forced.

By contrast, weakly admissible phases exhibit rigidity only up to the defect
control depth. Beyond this regime, genuine moduli of phase extensions may
appear, and different equivalence notions need not agree. In this section we
formalize the relevant equivalence notions, prove collapse results in the
strongly admissible setting, and isolate the precise mechanism by which
equivalence collapse can fail for weakly admissible phases.

\subsection*{Equivalence notions}

We begin by fixing the notions of equivalence relevant for algebraic phases.
These notions reflect different levels at which phase structure may be compared:
strictly as structured objects, up to controlled presentation data, or through
their categories of representations.

The definitions below distinguish between \emph{strong equivalence}, which
preserves the full phase structure, and \emph{weak equivalence}, which allows
flexibility beyond the defect-controlled regime. We also introduce filtered
representations, which encode phase data together with defect stratification and
serve as the categorical bridge to Morita-type notions of equivalence.

\begin{definition}
Two algebraic phases $\mathcal P$ and $\mathcal Q$ are \emph{strongly
equivalent} if they are isomorphic as structured phase objects, preserving
structural signature, interaction law, canonical filtration, termination
length, and defect data.
\end{definition}

\begin{definition}
Two algebraic phases $\mathcal P$ and $\mathcal Q$ are \emph{weakly
equivalent} if they become isomorphic after forgetting inessential
presentation data while retaining defect stratification up to the defect
control depth and filtration length.
\end{definition}

\begin{remark}
In the strongly admissible regime there is no residual extension data: the rigid
core coincides with the full phase, and all higher structure is forced by the
canonical filtration and defect axioms.
\end{remark}

\begin{definition}
Let $\mathcal P$ be an algebraic phase with canonical filtration
$\{\mathcal P^{(i)}\}$.
A \emph{filtered representation} of $\mathcal P$ consists of:
\begin{itemize}
\item a target category $\mathcal C$ equipped with a compatible filtration,
\item a representation functor $\rho:\mathcal P\to\mathcal C$,
\end{itemize}
such that $\rho(\mathcal P^{(i)})$ lands in the $i$-th filtration level of
$\mathcal C$ and preserves defect-induced relations on all controlled layers.
\end{definition}

\begin{remark}
Filtered representations are required to respect \emph{defect degree} on
controlled layers, not merely the interaction law.
This ensures that representation categories retain enough structure to detect
rigidity and defect stratification within the admissible regime.
\end{remark}

\begin{lemma}
\label{lem:rigid-core-detected}
Let $\mathcal P$ be a finite algebraic phase.
The rigid core $\mathcal P_{\mathrm{rig}}$ is functorially detectable from
the category of filtered representations of $\mathcal P$.
\end{lemma}

\begin{proof}
By definition of rigidity in Algebraic Phase Theory, an element of $\mathcal P$
is rigid if and only if it has defect degree zero. Equivalently, it lies in the
last filtration layer in which rigidity persists,
$\mathcal P^{(d(\mathcal P)-1)}$, and generates no further defect under phase
interaction.

Let $\rho:\mathcal P\to\mathcal C$ be a filtered representation. By definition, a
filtered representation preserves the canonical filtration:
\[
\rho(\mathcal P^{(i)}) \subseteq \mathcal C^{(i)} \qquad \text{for all } i.
\]
Equivalently, $\rho$ preserves defect degree: if $x\in \mathcal P$ has defect
degree $i$ (meaning $x\in \mathcal P^{(i)}$), then $\rho(x)\in \mathcal C^{(i)}$.

In particular, elements of defect degree zero in $\mathcal P$ are carried to
objects of defect degree zero in $\mathcal C$. Consequently, for every filtered
representation, the images of rigid elements are precisely those landing in the
filtration layer $\mathcal C^{(d(\mathcal P)-1)}$.

Consider now the category $\Rep_{\fil}(\mathcal P)$ of filtered representations
of $\mathcal P$. Within this category, the full subcategory generated by
defect degree zero elements is canonically defined and independent of choices.
Any equivalence of filtered representation categories is required to preserve
filtration levels, and therefore preserves this defect degree zero subcategory.

It follows that the defect degree zero part of the filtered representation
theory, and hence the rigid core $\mathcal P_{\mathrm{rig}}$, can be recovered
functorially from $\Rep_{\fil}(\mathcal P)$. This shows that the rigid core is
detectable purely from filtered representation data.
\end{proof}

\begin{lemma}
\label{lem:no-hidden-extensions}
Let $\mathcal P$ be a finite, \emph{strongly admissible} algebraic phase with
bounded defect.
Then the extension of $\mathcal P$ from its rigid core
$\mathcal P_{\mathrm{rig}}$ to the full phase is uniquely determined by the
canonical filtration and defect data.
\end{lemma}

\begin{proof}
Because $\mathcal P$ is finite and strongly admissible, its defect control depth
coincides with its filtration length:
\[
d(\mathcal P)=k(\mathcal P).
\]
In particular, $\mathcal P$ admits a finite canonical filtration
\[
\mathcal P = \mathcal P^{(0)} \supseteq \mathcal P^{(1)} \supseteq \cdots
\supseteq \mathcal P^{(d(\mathcal P))} = 0,
\]
whose structure is governed by the defect data of $\mathcal P$.

Strong admissibility means that defect control extends through the entire
filtration: for every index $1 \le i \le d(\mathcal P)$, the passage from
$\mathcal P^{(i)}$ to $\mathcal P^{(i-1)}$ is functorial and
presentation-independent.
Equivalently, each filtration layer $\mathcal P^{(i-1)}$ is generated
canonically from $\mathcal P^{(i)}$ by the defect-extension operations dictated
by the interaction law, and no uncontrolled extension freedom appears at any
stage.

More precisely, by Lemma~\ref{lem:filtration-generated}, for each
$1 \le i \le d(\mathcal P)$ the stratum $\mathcal P^{(i-1)}$ is obtained
functorially as the closure of $\mathcal P^{(i)}$ under the canonical
defect-extension operations extracted from the phase structure.
Since $\mathcal P$ is strongly admissible, these operations are
presentation-independent at every depth.

The assumption of bounded defect ensures that at each extension step only
finitely many defect constraints and relations are imposed, so the canonical
construction is finite and fixed.
In particular, no independent extension parameters arise beyond those dictated
by the defect data.

It follows that once $\mathcal P^{(i)}$ is fixed, the structure of
$\mathcal P^{(i-1)}$ is uniquely determined.
Starting from the rigid core
\[
\mathcal P_{\mathrm{rig}} = \mathcal P^{(d(\mathcal P)-1)},
\]
and proceeding inductively upward through the filtration, the entire phase
$\mathcal P$ is uniquely reconstructed from its rigid core together with the
canonical filtration and defect data.
Hence no nontrivial extension freedom remains beyond the rigid core.
\end{proof}

\begin{remark}
For weakly admissible phases with $k(\mathcal P)>d(\mathcal P)$, the conclusion
of Lemma~\ref{lem:no-hidden-extensions} may fail.
Beyond the defect control depth, genuine extension freedom can appear.
\end{remark}

\begin{definition}
Two phases are \emph{Morita-type equivalent} if their categories of filtered
representations are equivalent via filtration-preserving functors.
\end{definition}

\subsection*{Equivalence collapse for strongly admissible phases}

We now turn to the central rigidity results of this section. The key observation
is that for strongly admissible phases with bounded defect, there is no hidden
extension data beyond the rigid core. As a result, any equivalence that preserves
defect stratification on the controlled layers must preserve the entire phase
structure.

We first analyze weak equivalence for weakly admissible phases. We show that it
is completely controlled by the strongly admissible core together with residual
extension data below the defect control depth. We then specialize to the
strongly admissible regime. In this setting the residual freedom disappears and
weak equivalence collapses to strong equivalence.

Finally, we show that Morita type equivalence defined via equivalence of filtered
representation categories also collapses to strong equivalence in this regime.
As a consequence, finite strongly admissible phases admit no genuinely distinct
models with the same filtered representation theory.

\begin{theorem}
\label{thm:equivalence-reduction}
Let $\mathcal P$ and $\mathcal Q$ be weakly admissible algebraic phases.
Then $\mathcal P$ and $\mathcal Q$ are weakly equivalent if and only if:
\begin{enumerate}[label=(\roman*)]
\item their strongly admissible cores
      $\mathcal P^{(d(\mathcal P))}$ and $\mathcal Q^{(d(\mathcal Q))}$
      are strongly equivalent, and
\item their extension data beyond the defect control depth agree up to canonical
      isomorphism.
\end{enumerate}
\end{theorem}

\begin{proof}
Suppose first that $\mathcal P$ and $\mathcal Q$ are weakly equivalent.
By definition of weak equivalence, there is an identification that preserves
the structural signature and the defect stratification on all controlled layers,
i.e.\ on $\mathcal P^{(i)}$ and $\mathcal Q^{(i)}$ for every $i\le d(\mathcal P)$.
In particular, weak equivalence forces
\[
d(\mathcal P)=d(\mathcal Q)\qquad\text{and}\qquad k(\mathcal P)=k(\mathcal Q),
\]
so the controlled depth and filtration length agree on both sides.

It follows that the cores
\[
\mathcal P^{(d(\mathcal P))}\qquad\text{and}\qquad \mathcal Q^{(d(\mathcal Q))}
\]
are identified as structured algebraic phases.  Since defect propagation within
these cores is functorial and presentation-independent by definition of the
control depth, this identification preserves the full phase structure on the
core.  Hence the cores are strongly equivalent, proving (i).

Any remaining freedom in a weak equivalence can occur only in the layers
\emph{strictly beyond} the defect control depth, i.e.\ in strata
$\mathcal P^{(i)}$ with $i>d(\mathcal P)$.
These layers encode extension data not forced by defect propagation.
Thus weak equivalence determines this residual extension data only up to
canonical isomorphism, establishing (ii).

\smallskip\noindent
Conversely, suppose (i) and (ii) hold.  Thus the cores
$\mathcal P^{(d(\mathcal P))}$ and $\mathcal Q^{(d(\mathcal Q))}$ are strongly
equivalent, and the residual extension data in strata with $i>d(\mathcal P)$
agree up to canonical isomorphism.

By strong equivalence of the cores, the controlled part of each phase (up to
depth $d(\mathcal P)=d(\mathcal Q)$) matches.  Using the assumed compatibility
of the residual extension data, the deeper strata can also be identified.
Putting these identifications together yields a weak equivalence
$\mathcal P\simeq\mathcal Q$ preserving the structural signature, interaction
law, and defect stratification on all controlled layers.

Hence $\mathcal P$ and $\mathcal Q$ are weakly equivalent if and only if (i)
and (ii) hold.
\end{proof}

\begin{theorem}
\label{thm:weak-equivalence-classification}
Let $\mathcal P$ and $\mathcal Q$ be weakly admissible algebraic phases (of the
same structural signature). Then $\mathcal P$ and $\mathcal Q$ are weakly equivalent if and only if:
\begin{enumerate}[label=(\roman*)]
\item the defect control depths and filtration lengths agree,
\[
d(\mathcal P)=d(\mathcal Q)
\qquad\text{and}\qquad
k(\mathcal P)=k(\mathcal Q),
\]
\item the strongly admissible cores
\[
\mathcal P^{(d(\mathcal P))}\quad\text{and}\quad \mathcal Q^{(d(\mathcal Q))}
\]
are strongly equivalent (as structured phases), and
\item the remaining filtration layers below the control depth are compatible in
the sense that the non-functorial extension data needed to reconstruct
$\mathcal P^{(i-1)}$ from $\mathcal P^{(i)}$ for $i<d(\mathcal P)$ matches that
for $\mathcal Q$ up to canonical isomorphism.
\end{enumerate}
\end{theorem}

\begin{proof}
We prove both directions.

\medskip
\noindent ($\Rightarrow$). Assume $\mathcal P$ and $\mathcal Q$ are weakly equivalent. Unpacking the
definition, this means there exists an identification of $\mathcal P$ and
$\mathcal Q$ after forgetting inessential presentation data, but \emph{retaining}
the structural signature and the canonical defect stratification up to the defect
control depth, together with the associated filtration bookkeeping.

\smallskip
\noindent\emph{(i) Equality of $d$ and $k$.}
Since weak equivalence is required to preserve the canonical filtration
\emph{as a stratified object up to the controlled regime} (and to preserve
filtration length as part of the admissible datum), it sends filtration strata
to corresponding strata. In particular, the index at which defect propagation
ceases to be functorially controlled is an invariant of the weak equivalence
class. Hence the defect control depths agree:
\[
d(\mathcal P)=d(\mathcal Q).
\]
Likewise, weak equivalence preserves termination/filtration length data, so
\[
k(\mathcal P)=k(\mathcal Q).
\]
This establishes (i).

\smallskip
\noindent\emph{(ii) Strong equivalence of the strongly admissible cores.}
By Proposition~\ref{prop:strong-core}, the stratum $\mathcal P^{(d(\mathcal P))}$
is the \emph{strongly admissible core} of $\mathcal P$, and similarly for
$\mathcal Q$. The definition of weak equivalence includes preservation of defect
stratification and defect-induced relations on all controlled layers; in
particular, it identifies the controlled part of the filtration \emph{functorially}.
But on $\mathcal P^{(d(\mathcal P))}$ and $\mathcal Q^{(d(\mathcal Q))}$ there is
\emph{no residual extension freedom} by definition of defect control depth and
by the characterization of the strong core: all structure in these cores is
forced by the interaction law and canonical defect propagation.

Consequently, the weak equivalence induces an isomorphism of structured phase
objects
\[
\mathcal P^{(d(\mathcal P))}\;\cong\;\mathcal Q^{(d(\mathcal Q))}
\]
preserving the structural signature, interaction law, canonical filtration
restricted to the controlled regime, and controlled defect data. That is exactly
strong equivalence of the cores. This proves (ii).

\smallskip
\noindent\emph{(iii) Compatibility of the non-functorial extension data below depth $d$.}
By definition of defect control depth, for indices $i<d(\mathcal P)$ the passage
from $\mathcal P^{(i)}$ to $\mathcal P^{(i-1)}$ may involve \emph{residual
non-functorial choices}: there can be multiple inequivalent ways of realizing the
lower stratum consistent with the same higher stratum and the same controlled
defect data. These choices are precisely what we are calling the
\emph{non-functorial extension data}.

A weak equivalence is allowed to forget inessential presentation data, but it is
\emph{not} allowed to change the weak equivalence class of these extension
choices: it must match the remaining layers in a way compatible with the phase
structure (up to the permitted notion of canonicity). Therefore the extension
data occurring in $\mathcal P$ and $\mathcal Q$ below the control depth must
match up to canonical isomorphism. This is exactly condition (iii).

\medskip
\medskip
\noindent($\Leftarrow$).
Now assume conditions (i)–(iii). We explain how to assemble a weak equivalence
$\mathcal P \sim \mathcal Q$.

The strategy is as follows. We first identify the strongly admissible cores of
$\mathcal P$ and $\mathcal Q$, which agree by hypothesis. Because defect
propagation is functorial above the control depth, this identification uniquely
determines all filtration layers in the controlled regime. Below the control
depth, functoriality fails in general, so additional extension data must be
specified. Here we use the hypothesis that these weak extension profiles agree
up to canonical isomorphism. Combining these two ingredients allows us to match
the entire canonical filtrations of $\mathcal P$ and $\mathcal Q$ in a way that
preserves the phase structure up to weak equivalence.

\smallskip
By condition (ii), there is a strong equivalence
\[
\phi_d:\mathcal P^{(d(\mathcal P))}\;\xrightarrow{\ \cong\ }\;
\mathcal Q^{(d(\mathcal Q))}.
\]
By condition (i), the defect control depths and filtration lengths agree, so we
write
\[
d:=d(\mathcal P)=d(\mathcal Q),
\qquad
k:=k(\mathcal P)=k(\mathcal Q).
\]

For strata at and beyond the defect control depth, that is, for all $i\ge d$,
defect propagation is functorial and presentation independent. Equivalently,
once the strongly admissible core $\mathcal P^{(d)}$ is fixed, every higher
filtration stratum is uniquely determined by iterating the canonical
defect extension and closure operations prescribed by the structural signature.
Applying the same functorial constructions on both sides, the identification
$\phi_d$ therefore extends canonically to isomorphisms
\[
\phi_i:\mathcal P^{(i)}\;\xrightarrow{\ \cong\ }\;\mathcal Q^{(i)}
\qquad\text{for all } i\ge d.
\]
No choices are involved in this extension, precisely because we are working
entirely within the defect controlled regime. For indices $i<d$, the situation is different. The passage from
$\mathcal P^{(i)}$ to $\mathcal P^{(i-1)}$ is not, in general, functorially
forced. Additional non canonical extension data may be required. Condition (iii)
asserts that the extension data needed to reconstruct $\mathcal P^{(i-1)}$ from
$\mathcal P^{(i)}$ matches the corresponding extension data for $\mathcal Q$ up
to canonical isomorphism. Using this compatibility, we choose identifications
\[
\phi_i:\mathcal P^{(i)}\;\xrightarrow{\ \cong\ }\;\mathcal Q^{(i)}
\qquad\text{for each } i<d,
\]
so that each extension step
$\mathcal P^{(i)}\rightsquigarrow\mathcal P^{(i-1)}$ corresponds to
$\mathcal Q^{(i)}\rightsquigarrow\mathcal Q^{(i-1)}$ in a manner compatible with
the interaction law, filtration structure, and defect data, up to the permitted
canonical isomorphisms.

In this way we obtain compatible identifications $\phi_i$ for all filtration
levels $i=0,1,\dots,k$, such that the structural signature, interaction law,
canonical filtration, and controlled defect stratification are respected, and
the remaining non functorial extension choices below the control depth are
matched as allowed by weak equivalence. This is exactly the data required to
identify $\mathcal P$ and $\mathcal Q$ after forgetting inessential presentation
choices while retaining the admissible phase structure. Hence $\mathcal P$ and $\mathcal Q$ are weakly equivalent.

\end{proof}

\begin{corollary}
For finite, strongly admissible phases with bounded defect,
Morita-type equivalence implies strong equivalence.
\end{corollary}

\begin{proof}
Let $\mathcal P$ and $\mathcal Q$ be finite, strongly admissible algebraic phases
with bounded defect, and suppose that they are Morita-type equivalent. By
definition, this means that their categories of filtered representations are
equivalent via a filtration-preserving equivalence
\[
\Rep_{\fil}(\mathcal P)\;\simeq\;\Rep_{\fil}(\mathcal Q).
\]

The argument proceeds as follows. First, we observe that a filtration-preserving
equivalence of filtered representation categories necessarily preserves defect
degree, and therefore identifies the rigid cores of the two phases. Second, we
use strong admissibility and bounded defect to show that once the rigid core is
fixed, there is a unique way to reconstruct the full phase. Combining these two
observations shows that Morita-type equivalence forces an isomorphism of the
underlying phase objects.

Because the equivalence of representation categories preserves the filtration, it
carries objects and morphisms in the $i$th filtration level on the
$\mathcal P$-side to objects and morphisms in the $i$th filtration level on the
$\mathcal Q$-side. In particular, it preserves defect degree. Rigid elements are
exactly those of defect degree zero, so the equivalence identifies the defect-zero
part of the filtered representation theory on both sides. By
Lemma~\ref{lem:rigid-core-detected}, this defect-zero theory determines the rigid
core functorially as a structured phase object. Hence the Morita-type equivalence
induces an isomorphism
\[
\mathcal P_{\mathrm{rig}}\;\cong\;\mathcal Q_{\mathrm{rig}}
\]
of rigid phase objects. Since $\mathcal P$ and $\mathcal Q$ are strongly admissible, the rigid core
coincides with the entire defect-controlled structure of each phase.
By Lemma~\ref{lem:no-hidden-extensions}, there is no nontrivial extension data
beyond the rigid core: once $\mathcal P_{\mathrm{rig}}$ is fixed, the full phase
$\mathcal P$ is uniquely determined, and similarly for $\mathcal Q$. Because the Morita-type equivalence identifies the rigid cores
\[
\mathcal P_{\mathrm{rig}}\;\cong\;\mathcal Q_{\mathrm{rig}},
\]
and there are no further weak layers to reconstruct, this identification extends
uniquely to an isomorphism of phases
\[
\mathcal P \cong \mathcal Q.
\]
In particular, the structural signature, interaction law, canonical filtration,
and defect data all coincide. This is exactly strong equivalence.

\end{proof}

\begin{remark}
Filtered representation theory behaves very differently in the strongly and
weakly admissible regimes.

For finite, strongly admissible algebraic phases, filtered representation
theory is a complete invariant: no two non-isomorphic phases share the same
filtered representation category.

For weakly admissible phases, this rigidity holds only within the
defect-controlled regime. Beyond the defect control depth, Morita-type
equivalence may preserve filtered representation theory while allowing
non-isomorphic phase extensions. These residual extension parameters are
intrinsic and will be analyzed in later work.
\end{remark}

\section{Intrinsic Invariants of Phases}

A central feature of Algebraic Phase Theory is that much of the structure of a
phase is forced rather than chosen.
In particular, numerical and stratified data extracted from the canonical
filtration and rigidity structure are intrinsic and do not depend on
presentation.

The extent to which these invariants are fully determined depends on the
strength of the admissible input data.
For strongly admissible phases, all invariants listed below are rigid and
complete.
For weakly admissible phases, the same invariants remain well-defined on the
controlled defect regime, but may admit additional freedom beyond the defect
control depth.

This section records the intrinsic invariants that are forced by phase
interaction and explains their functorial behaviour under equivalence. 
The following invariance statement is immediate from the definitions, 
but we record it explicitly since these quantities will serve as basic 
invariants throughout the remainder of the series.

\medskip

\begin{theorem}
\label{thm:intrinsic-invariants}
Let $\mathcal P$ and $\mathcal Q$ be \emph{strongly admissible} algebraic phases
that are equivalent as structured phase objects.
Then $\mathcal P$ and $\mathcal Q$ have the same termination length, defect
rank, boundary depth, and signature complexity.
\end{theorem}

\begin{proof}
An equivalence of algebraic phases preserves the structural signature and the
canonical filtration by definition.
Since the termination length $L(\mathcal P)$ is the length of the canonical
filtration
\[
\mathcal P = \mathcal P^{(0)} \supseteq \cdots \supseteq \mathcal P^{(k(\mathcal P))}=0,
\]
it is invariant under phase equivalence.

The defect rank is determined by the maximal rank (or dimension) of the
non-rigid filtration layers $\mathcal P^{(i)}/\mathcal P^{(i+1)}$, and is
therefore preserved as well. Equivalence preserves the rigid core
$\mathcal P_{\mathrm{rig}}=\mathcal P^{(k(\mathcal P)-1)}$, and hence induces a canonical
equivalence of quotients
\[
\partial\mathcal P \;\cong\; \mathcal P/\mathcal P_{\mathrm{rig}},
\]
showing that boundary depth is invariant. Finally, signature complexity depends only on the structural signature of the
phase, which is part of the preserved data under equivalence.
Thus all listed invariants agree for $\mathcal P$ and $\mathcal Q$.
\end{proof}

\begin{remark}
For weakly admissible phases, the same invariants are canonically defined on the
controlled defect subphase $\mathcal P^{(d(\mathcal P))}$.
Beyond the defect control depth, termination length and boundary structure may
admit additional extension freedom and are therefore not expected to be rigid
invariants.
\end{remark}

\medskip

\begin{theorem}
\label{thm:boundary-morita}
Let $\mathcal P$ and $\mathcal Q$ be \emph{strongly admissible} algebraic phases.
If $\mathcal P$ and $\mathcal Q$ are Morita-type equivalent, then their
boundaries are canonically equivalent:
\[
\partial\mathcal P \;\simeq\; \partial\mathcal Q.
\]
\end{theorem}

\begin{proof}
A Morita-type equivalence between $\mathcal P$ and $\mathcal Q$ induces an
equivalence
\[
\Rep_{\fil}(\mathcal P)\;\simeq\;\Rep_{\fil}(\mathcal Q)
\]
of their categories of filtered representations.
By definition, such an equivalence preserves filtration levels and hence
preserves defect degree.

Rigid elements are precisely those of defect degree zero.
By Lemma~\ref{lem:rigid-core-detected}, the defect-zero part of the filtered
representation theory detects the rigid core functorially.
It follows that the induced equivalence identifies the rigid cores
$\mathcal P_{\mathrm{rig}}$ and $\mathcal Q_{\mathrm{rig}}$ up to canonical
equivalence.

The boundary of a phase is defined as the quotient by its rigid core,
$\partial\mathcal P = \mathcal P/\mathcal P_{\mathrm{rig}}$.
Since formation of this quotient is functorial with respect to rigid-core
preserving morphisms, the equivalence of rigid cores induces a canonical
equivalence
\[
\partial\mathcal P \;\simeq\; \partial\mathcal Q.
\]
\end{proof}

\begin{theorem}
\label{thm:boundary-localization}
Let $\mathcal P$ be a weakly admissible algebraic phase, and write
$d:=d(\mathcal P)$.
Let $F:\mathcal P\to \mathcal R$ be any phase morphism (or any functorial
invariant) such that the restriction of $F$ to the strongly admissible core is
trivial in the relevant sense (equivalently, $F$ does not distinguish elements
of $\mathcal P^{(d)}$).
Then $F$ factors (uniquely) through the quotient map
\[
\pi:\mathcal P \longrightarrow \mathcal P/\mathcal P^{(d)}.
\]
In particular, any such $F$ vanishes on $\mathcal P^{(d)}$.
\end{theorem}

\begin{proof}
Let $\mathcal P$ be a weakly admissible algebraic phase and write
$d := d(\mathcal P)$.
By definition of the defect control depth, the subphase
$\mathcal P^{(d)}$ consists exactly of those filtration layers whose defect
behaviour is functorially forced by the interaction law and the defect axioms.
In particular, $\mathcal P^{(d)}$ contains no non-functorial extension data.

Let $F:\mathcal P \to \mathcal R$ be a phase morphism (or functorial invariant)
whose restriction to $\mathcal P^{(d)}$ is trivial in the relevant sense.
Equivalently, $F$ does not distinguish between elements of $\mathcal P^{(d)}$. Since $F$ identifies all elements of $\mathcal P^{(d)}$, it is constant on that
subphase. By assumption, $F$ identifies all elements of $\mathcal P^{(d)}$, so it is constant on that subphase and therefore descends to the quotient that collapses $\mathcal P^{(d)}$. By the universal property of the quotient, there therefore exists a
unique morphism
\[
\widetilde F:\mathcal P/\mathcal P^{(d)} \longrightarrow \mathcal R
\]
such that
\[
F = \widetilde F \circ \pi,
\]
where $\pi:\mathcal P \to \mathcal P/\mathcal P^{(d)}$ denotes the canonical
quotient map. This establishes the claimed factorization property. In particular, any such
$F$ vanishes on the strongly admissible core $\mathcal P^{(d)}$.
\end{proof}

\begin{remark}
For weakly admissible phases, Morita-type equivalence still induces equivalence
of boundaries \emph{up to the controlled defect depth}.
Beyond this regime, boundary strata may contain genuine extension parameters
and need not be invariant under Morita-type equivalence.
\end{remark}

\section{The Category and 2-Category of Phases}

Algebraic Phase Theory is not only concerned with individual phases, but also
with the maps between them and the ways in which such maps can vary.
Because phases carry intrinsic filtration and defect data, morphisms are highly
constrained: any meaningful notion of morphism must respect the canonical
filtration and defect stratification. The categorical behaviour of phases depends crucially on the strength of the 
admissible input.
For strongly admissible phases, defect propagation is fully controlled and no
independent extension choices remain.
In this regime, both morphisms and higher morphisms are rigidly determined.

For weakly admissible phases, by contrast, rigidity holds only up to the defect
control depth.
Below this depth, genuine extension freedom may appear, and categorical
structures must allow controlled flexibility. In this section we make these distinctions explicit.
We first define the category formed by strongly admissible phases and explain
why this restriction is necessary for categorical rigidity.
We then introduce a $2$-categorical refinement in which $2$-morphisms encode
the residual flexibility that cannot be captured at the level of phase
morphisms alone.

\subsection*{The category of phases}

We begin by isolating the categorical structure associated to
\emph{strongly admissible} algebraic phases.
In this regime, defect control extends through the entire filtration, and the
phase is fully determined by its rigid core together with the canonical defect
propagation rules. Restricting to strongly admissible phases ensures that phase morphisms are
themselves rigid objects: once their behaviour on the rigid core is fixed, no
additional extension choices are available.
This rigidity is essential for forming a well-behaved category with strictly
determined composition and identities.

Weakly admissible phases, by contrast, may carry uncontrolled extension data in
lower filtration layers.
Allowing such phases as objects would introduce morphisms whose behaviour is not
forced by defect data alone.
For this reason, the basic category of phases is defined using strongly
admissible objects.

\begin{definition}
Let $\mathsf{C}$ denote the category whose objects are \emph{strongly admissible}
algebraic phases satisfying the axioms of Algebraic Phase Theory, and whose
morphisms are \emph{phase morphisms}, that is, structure-preserving maps which
preserve the structural signature, interaction law, canonical filtration,
termination length, and defect data.
Composition and identities are given by ordinary composition and identity maps.
\end{definition}

\begin{proposition}
\label{prop:non-full-inclusion}
The inclusion functor
\[
\iota:\mathsf{APT}_{\mathrm{strong}}\hookrightarrow
\mathsf{APT}_{\mathrm{weak}}
\]
from strongly admissible phases to weakly admissible phases is faithful but not
full.
\end{proposition}

\begin{proof}
We first show that $\iota$ is faithful.  Let $\mathcal P,\mathcal Q$ be strongly
admissible phases, and let
\[
F,G:\mathcal P\to\mathcal Q
\]
be morphisms in $\mathsf{APT}_{\mathrm{strong}}$ such that
$\iota(F)=\iota(G)$ as morphisms in $\mathsf{APT}_{\mathrm{weak}}$.  This means
simply that $F$ and $G$ are the same underlying phase morphism
$\mathcal P\to\mathcal Q$ when regarded in the ambient weakly admissible
setting.  In particular, they agree on the rigid core $\mathcal P_{\mathrm{rig}}$.
By strong morphism rigidity (Theorem~\ref{thm:morphism-rigidity}), a morphism
between strongly admissible phases is uniquely determined by its restriction to
the rigid core.  Hence $F=G$.  Therefore $\iota$ is faithful.

\smallskip We now show that $\iota$ is not full.  Fullness would require that for any
strongly admissible phases $\mathcal P,\mathcal Q$, every phase morphism
\[
H:\iota(\mathcal P)\to\iota(\mathcal Q)
\]
in $\mathsf{APT}_{\mathrm{weak}}$ arises from a morphism
$\mathcal P\to\mathcal Q$ in $\mathsf{APT}_{\mathrm{strong}}$. This fails because weak admissibility allows additional morphisms supported on
filtration layers strictly below the defect control depth.  When a weakly
admissible phase $\mathcal R$ satisfies $k(\mathcal R)>d(\mathcal R)$, its lower
layers carry extension data not functorially determined by defect propagation.
Morphisms in $\mathsf{APT}_{\mathrm{weak}}$ are required to be rigid only on the
controlled regime; on uncontrolled layers they may vary independently.

Such morphisms cannot be induced from $\mathsf{APT}_{\mathrm{strong}}$, where
defect control extends through the entire filtration and no independent
extension choices exist.  Hence the inclusion functor $\iota$ fails to be full.
Combined with faithfulness, this shows that $\iota$ is faithful but not full.
\end{proof}

\begin{remark}
Restricting $\mathsf{C}$ to strongly admissible phases ensures that all defect
layers are functorially generated and that no uncontrolled extension freedom
appears.
This restriction is essential for full categorical rigidity and strict
2-categorical behaviour.
\end{remark}

\subsection*{2-morphisms and 2-categorical structure}

Even in the strongly admissible regime, equality of phase morphisms can be too
rigid a notion.
Two morphisms may agree completely on the rigid, defect-controlled structure of
a phase while differing only by higher coherence or obstruction data that does
not affect defect stratification.

Such differences are invisible at the level of $1$-morphisms.
To capture this controlled flexibility, we introduce $2$-morphisms between
phase morphisms.
These $2$-morphisms record variation that respects the canonical filtration and
defect degrees, without introducing uncontrolled mixing of layers. This leads naturally to a $2$-categorical refinement of the category of phases,
in which objects are strongly admissible phases, $1$-morphisms are phase
morphisms, and $2$-morphisms encode filtration-compatible coherence data.

\begin{definition}
Let $F,G:\mathcal P\to\mathcal Q$ be morphisms in $\mathsf{C}$.
A \emph{2-morphism} $\eta:F\Rightarrow G$ is a natural transformation that is
compatible with the canonical filtration, in the sense that it does not mix
defect levels.
Such 2-morphisms are called \emph{filtration-compatible natural
transformations}.
\end{definition}

Filtration-compatible $2$-morphisms encode the fact that phase morphisms may
agree rigidly on the defect-controlled core while differing only by higher
obstruction or extension data, which cannot be detected at the level of
$1$-morphisms alone.

\begin{theorem}
Algebraic phases, phase morphisms, and filtration-compatible natural
transformations assemble into a strict $2$-category.
\end{theorem}

\begin{proof}
We define vertical and horizontal compositions of filtration-compatible natural
transformations and verify the strict associativity, unit axioms, and the
interchange law by componentwise equalities.

By definition, objects and $1$-morphisms form the category $\mathsf{C}$.
Let $F,G:\mathcal P\to\mathcal Q$ be phase morphisms.  
A $2$-morphism $\eta:F\Rightarrow G$ is a natural transformation whose
components
\[
\eta_x : F(x) \longrightarrow G(x)
\]
are filtration-compatible, meaning that they preserve defect degree and do not
mix filtration levels.

Let $\eta:F\Rightarrow G$ and $\theta:G\Rightarrow H$ be $2$-morphisms.
Define their vertical composite $(\theta\circ\eta):F\Rightarrow H$ componentwise by
\[
(\theta\circ\eta)_x := \theta_x\circ \eta_x.
\]
Filtration-compatibility is preserved because composition in the target does
not change filtration level when each component respects the filtration.

Vertical associativity and unit axioms follow immediately from associativity
and identities in the target hom-sets.

For horizontal composition, let $F_1,F_2:\mathcal P\to\mathcal Q$ and
$G_1,G_2:\mathcal Q\to\mathcal R$ be phase morphisms, and let
$\eta:F_1\Rightarrow F_2$ and $\theta:G_1\Rightarrow G_2$ be filtration-compatible
natural transformations.
Define the horizontal composite
\[
\theta\ast\eta : G_1\circ F_1 \Rightarrow G_2\circ F_2
\]
by the standard whiskering formula,
\[
(\theta\ast\eta)_x := \theta_{F_2(x)}\circ G_1(\eta_x).
\]
Filtration-compatibility is preserved because $G_1$ preserves the filtration
and $\eta_x$ does not mix defect levels.

Horizontal associativity, unit axioms, and the interchange law are verified
componentwise using functoriality and associativity in the target hom-sets.
All identities hold strictly.
Therefore algebraic phases, phase morphisms, and filtration-compatible natural
transformations form a strict $2$-category.
\end{proof}

\begin{remark}
The strictness of the resulting $2$-category reflects full defect control in
the strongly admissible regime.
Once defect degree and canonical filtration are fixed, no nontrivial higher
coherence choices remain.

For weakly admissible phases, the same constructions define a \emph{partial}
2-categorical structure that is strict only up to the defect control depth.
Beyond this regime, higher coherence data may appear, reflecting genuine
extension freedom rather than a failure of functoriality.
\end{remark}

\section{Completion as Reflective Localization}

A recurring theme in Algebraic Phase Theory is that non-rigid behaviour is
organized into canonical defect strata, and that the \emph{fully rigid}, or
\emph{complete}, object attached to a phase is forced rather than chosen.
Categorically, this forced rigidification is best expressed by an adjunction:
completion is left adjoint to the inclusion of complete phases.
Equivalently, completion exhibits the complete phases as a reflective
subcategory.

This phenomenon holds in full generality for \emph{strongly admissible} phases,
where defect propagation is globally controlled and the canonical filtration
terminates at the defect control depth.
In this regime, completion is an intrinsic categorical operation rather than an
auxiliary construction.

\medskip
\noindent
Informally, the content of this section is that for strongly admissible phases,
finite termination and defect control ensure that every phase admits a uniquely
determined fully rigid refinement, and that this refinement is universal among
maps to complete phases.
In particular, completion adds exactly the relations forced by defect
propagation and nothing extraneous.

\subsection*{Completion of strongly admissible phases}

For strongly admissible phases, defect propagation determines a canonical notion
of closure: one may freely add all relations forced by the defect axioms and the
canonical filtration, without introducing new defect strata.
Phases for which no further such relations can be added are called complete.
In particular, completeness is characterized by the fact that the canonical
completion map is already an isomorphism.
We now formalize this notion and show that every strongly admissible phase admits
a canonical completion, which is functorial with respect to phase morphisms.

\begin{definition}
An algebraic phase $\mathcal P$ is \emph{complete} if it is terminal among
strongly admissible phases equipped with a phase morphism from $\mathcal P$
that does not change the rigid core and does not introduce new defect strata.
Equivalently, $\mathcal P$ is complete if every phase morphism
$\mathcal P\to\mathcal Q$ between strongly admissible phases that is an
isomorphism on the rigid core and preserves termination length is already an
isomorphism of phases.
\end{definition}

Let $\mathsf{C}^{\mathrm{cpl}}\subseteq\mathsf{C}$ denote the full subcategory
of strongly admissible complete phases, and let
\[
I:\mathsf{C}^{\mathrm{cpl}}\hookrightarrow \mathsf{C}
\]
be the inclusion functor.

\begin{remark}
This definition formalizes the idea that a complete phase has no missing
relations beyond those forced by its defect axioms and canonical filtration.
Once complete, a strongly admissible phase admits no further nontrivial
compatible extensions.
\end{remark}

\begin{proposition}
\label{prop:completion-fixed-points}
Let $\mathcal P$ be a strongly admissible algebraic phase for which a completion
$\Comp(\mathcal P)$ exists, with unit map
\[
\eta_{\mathcal P}:\mathcal P\longrightarrow \Comp(\mathcal P).
\]
Then $\mathcal P$ is complete if and only if $\eta_{\mathcal P}$ is an
isomorphism.  Equivalently,
\[
\mathcal P \text{ is complete } \quad\Longleftrightarrow\quad \Comp(\mathcal P)\cong\mathcal P.
\]
\end{proposition}

\begin{proof}
Suppose first that $\mathcal P$ is complete.
By definition, this means that $\mathcal P$ is terminal among strongly
admissible phases equipped with a phase morphism from $\mathcal P$ that
preserves the rigid core and does not introduce new defect strata.
Equivalently, for any complete phase $\mathcal Q$ and any phase morphism
$f:\mathcal P\to\mathcal Q$ satisfying these conditions, there exists a unique
isomorphism $\mathcal Q\cong\mathcal P$ compatible with $f$. By construction, $\Comp(\mathcal P)$ is a complete phase equipped with a phase
morphism
\[
\eta_{\mathcal P}:\mathcal P\longrightarrow\Comp(\mathcal P)
\]
that preserves the rigid core and termination length.
Applying the defining terminal property of completeness to this morphism,
it follows that $\eta_{\mathcal P}$ must be an isomorphism.
Thus $\Comp(\mathcal P)\cong\mathcal P$.

Conversely, suppose that $\eta_{\mathcal P}:\mathcal P\to\Comp(\mathcal P)$ is
an isomorphism.
Let $\mathcal Q$ be any complete strongly admissible phase and let
$f:\mathcal P\to\mathcal Q$ be a phase morphism preserving the rigid core and
termination length.
By the universal property of completion, there exists a unique morphism
\[
\overline f:\Comp(\mathcal P)\longrightarrow\mathcal Q
\]
such that
\[
f=\overline f\circ\eta_{\mathcal P}.
\]
Since $\eta_{\mathcal P}$ is an isomorphism, this factorization identifies $f$
uniquely with a morphism out of $\mathcal P$ itself.
In particular, the universal mapping property satisfied by
$\Comp(\mathcal P)$ is already satisfied by $\mathcal P$. Therefore $\mathcal P$ is terminal among strongly admissible phases receiving a
morphism from $\mathcal P$ that preserves the rigid core and defect structure.
Equivalently, no further relations forced by defect propagation can be added to
$\mathcal P$ without changing its phase structure.
Hence $\mathcal P$ is complete.
\end{proof}

Recall that for a strongly admissible algebraic phase $\mathcal P$, when it
exists, $\Comp(\mathcal P)$ denotes a chosen object of
$\mathsf{C}^{\mathrm{cpl}}$ equipped with a phase morphism
\[
\eta_{\mathcal P}:\mathcal P\to\Comp(\mathcal P).
\]
The defining feature of $\Comp(\mathcal P)$ is that it is initial among complete
phases under $\mathcal P$.

\begin{proposition}
\label{prop:completion-functor}
On the full subcategory of strongly admissible phases, the assignment
\[
\mathcal P \longmapsto \Comp(\mathcal P)
\]
extends canonically to a functor
\[
\Comp:\mathsf{C}\longrightarrow\mathsf{C}^{\mathrm{cpl}}.
\]
\end{proposition}

\begin{proof}
We define $\Comp$ on morphisms and verify that it preserves identities and
composition. Let $g:\mathcal P\to\mathcal P'$ be a morphism in $\mathsf{C}$.
Consider the composite
\[
f := \eta_{\mathcal P'}\circ g:\mathcal P\longrightarrow\Comp(\mathcal P').
\]
Since $\Comp(\mathcal P')$ is complete, the universal property of
$(\Comp(\mathcal P),\eta_{\mathcal P})$ applies. Hence there exists a unique
morphism
\[
\Comp(g):\Comp(\mathcal P)\longrightarrow\Comp(\mathcal P')
\]
such that the triangle commutes:
\[
\eta_{\mathcal P'}\circ g \;=\; \Comp(g)\circ \eta_{\mathcal P}.
\]
This defines $\Comp$ on morphisms.

If $g=\id_{\mathcal P}$, then $\eta_{\mathcal P}\circ \id_{\mathcal P}=\eta_{\mathcal P}$.
Both $\Comp(\id_{\mathcal P})$ and $\id_{\Comp(\mathcal P)}$ are morphisms
$\Comp(\mathcal P)\to\Comp(\mathcal P)$ satisfying
\[
\eta_{\mathcal P} \;=\; (\;\cdot\;)\circ \eta_{\mathcal P}.
\]
By uniqueness in the universal property, $\Comp(\id_{\mathcal P})=\id_{\Comp(\mathcal P)}$.

Now let $g:\mathcal P\to\mathcal P'$ and $h:\mathcal P'\to\mathcal P''$ be morphisms
in $\mathsf{C}$. By definition of $\Comp(g)$ and $\Comp(h)$ we have
\[
\eta_{\mathcal P'}\circ g \;=\; \Comp(g)\circ\eta_{\mathcal P}
\qquad\text{and}\qquad
\eta_{\mathcal P''}\circ h \;=\; \Comp(h)\circ\eta_{\mathcal P'}.
\]
Composing gives
\[
\eta_{\mathcal P''}\circ(h\circ g)
= \Comp(h)\circ\Comp(g)\circ\eta_{\mathcal P}.
\]
On the other hand, by definition of $\Comp(h\circ g)$ we also have
\[
\eta_{\mathcal P''}\circ(h\circ g)
= \Comp(h\circ g)\circ\eta_{\mathcal P}.
\]
Thus $\Comp(h\circ g)$ and $\Comp(h)\circ\Comp(g)$ coincide by uniqueness in the
universal property. Hence $\Comp$ preserves identities and composition.
\end{proof}

\begin{remark}
The universal property expresses that $\Comp(\mathcal P)$ is initial among
complete phases receiving a morphism from $\mathcal P$.
Intuitively, $\Comp(\mathcal P)$ is obtained from $\mathcal P$ by adding exactly
the relations forced by defect propagation and nothing extraneous.
\end{remark}

\subsection*{Universality and reflective localization}

The universal property of completion admits a precise categorical formulation.
We show that completion is left adjoint to the inclusion of complete phases,
so that complete phases form a reflective subcategory of strongly admissible
phases.

\begin{theorem}
\label{thm:completion-adjunction}
On the category of strongly admissible algebraic phases, the functor
\[
\Comp:\mathsf{C}\longrightarrow\mathsf{C}^{\mathrm{cpl}}
\]
is left adjoint to the inclusion
\[
I:\mathsf{C}^{\mathrm{cpl}}\hookrightarrow\mathsf{C}.
\]
\end{theorem}

\begin{proof}
We construct a natural bijection
\[
\Hom_{\mathsf{C}^{\mathrm{cpl}}}\bigl(\Comp(\mathcal P),\mathcal Q\bigr)
\;\cong\;
\Hom_{\mathsf{C}}\bigl(\mathcal P,I(\mathcal Q)\bigr)
\]
for every strongly admissible phase $\mathcal P\in\mathsf{C}$ and every complete
phase $\mathcal Q\in\mathsf{C}^{\mathrm{cpl}}$.

Let $\mathcal P\in\mathsf{C}$ and $\mathcal Q\in\mathsf{C}^{\mathrm{cpl}}$.
Given a morphism
\[
\alpha:\Comp(\mathcal P)\longrightarrow \mathcal Q
\]
in $\mathsf{C}^{\mathrm{cpl}}$, define
\[
\Phi_{\mathcal P,\mathcal Q}(\alpha)
:= \alpha\circ\eta_{\mathcal P}:\mathcal P\longrightarrow I(\mathcal Q).
\]
This gives a well-defined map
\[
\Phi_{\mathcal P,\mathcal Q}:
\Hom_{\mathsf{C}^{\mathrm{cpl}}}(\Comp(\mathcal P),\mathcal Q)
\longrightarrow
\Hom_{\mathsf{C}}(\mathcal P,I(\mathcal Q)).
\]

Conversely, let
\[
f:\mathcal P\longrightarrow I(\mathcal Q)
\]
be a morphism in $\mathsf{C}$, where $\mathcal Q$ is complete.  By the defining
universal property of $\Comp(\mathcal P)$, there exists a unique morphism
\[
\overline f:\Comp(\mathcal P)\longrightarrow \mathcal Q
\]
such that
\[
f=\overline f\circ\eta_{\mathcal P}.
\]
Define
\[
\Psi_{\mathcal P,\mathcal Q}(f):=\overline f.
\]

By construction, the maps $\Phi_{\mathcal P,\mathcal Q}$ and
$\Psi_{\mathcal P,\mathcal Q}$ are inverse to one another.  Indeed, if
$\alpha:\Comp(\mathcal P)\to\mathcal Q$, then uniqueness in the universal
property of completion implies
\[
\Psi_{\mathcal P,\mathcal Q}(\Phi_{\mathcal P,\mathcal Q}(\alpha))=\alpha.
\]
Similarly, if $f:\mathcal P\to I(\mathcal Q)$, then by definition
\[
\Phi_{\mathcal P,\mathcal Q}(\Psi_{\mathcal P,\mathcal Q}(f))
=\overline f\circ\eta_{\mathcal P}
=f.
\]

Naturality of this bijection in both $\mathcal P$ and $\mathcal Q$ follows
directly from functoriality of $\Comp$ and functoriality of composition in
$\mathsf{C}$.  Hence $\Comp$ is left adjoint to the inclusion
$I:\mathsf{C}^{\mathrm{cpl}}\hookrightarrow\mathsf{C}$.
\end{proof}

\begin{corollary}
\label{cor:completion-reflective}
The subcategory $\mathsf{C}^{\mathrm{cpl}}$ of complete phases is a reflective
subcategory of the category of strongly admissible phases.
\end{corollary}

\begin{proof}
By Theorem~\ref{thm:completion-adjunction}, the inclusion functor
\[
I:\mathsf{C}^{\mathrm{cpl}}\hookrightarrow\mathsf{C}
\]
admits a left adjoint
\[
\Comp:\mathsf{C}\longrightarrow\mathsf{C}^{\mathrm{cpl}}.
\]
By definition, a full subcategory is reflective precisely when its inclusion
functor has a left adjoint.  Since $\mathsf{C}^{\mathrm{cpl}}$ is a full
subcategory of $\mathsf{C}$ and $I$ admits the left adjoint $\Comp$, it follows
that $\mathsf{C}^{\mathrm{cpl}}$ is a reflective subcategory of the category of
strongly admissible phases.
\end{proof}

\begin{remark}[Weakly admissible phases]
For weakly admissible phases, completion exists only up to the defect control
depth.
In this regime, one obtains a \emph{partial completion} that is universal among
maps to phases that are complete up to the controlled layers.
The resulting construction defines a truncated adjunction rather than a full
reflective localization, reflecting the presence of genuine extension freedom
beyond the defect control depth.
\end{remark}

\section{Examples Across Flagship Phases}

This section illustrates how the categorical and rigidity results proved above
manifest in the principal classes of algebraic phases developed in
Papers~I-III~\cite{GildeaAPT1,GildeaAPT2,GildeaAPT3}.
Rather than introducing new constructions, we reinterpret familiar phase
geometries through the lens of morphism rigidity, equivalence collapse, and
boundary invariance.
Across all examples, strongly admissible phases exhibit full rigidity and
equivalence collapse, while weakly admissible phases display partial rigidity
confined to the defect-controlled regime.

\subsection*{Radical phase geometry}

In Algebraic Phase Theory~I~\cite{GildeaAPT1}, phases arise from radical and
nilpotent structures, with defect measuring the failure of additivity and the
canonical filtration coinciding with the radical filtration.
Finite termination reflects nilpotence of the underlying data. When the radical data is strongly admissible, the rigidity results of the
present paper have a concrete interpretation: any phase morphism is uniquely
determined by its action on the rigid core, corresponding to the maximal
additive (degree-zero) component.
Higher defect layers introduce no independent extension freedom, as they are
generated functorially from the radical structure.

In weakly admissible radical phases, rigidity persists only up to the defect
control depth.
Beyond this depth, higher radical layers admit genuine extension freedom,
corresponding to non-forced choices in the nilpotent hierarchy.
Categorically, this appears as partial morphism rigidity rather than full
determination by the rigid core. Boundaries classify residual phase behaviour modulo rigidity.
Accordingly, the boundary invariance results recover the classification of
radical phase types from~\cite{GildeaAPT1} in the strongly admissible regime,
while isolating the precise obstruction to full rigidity in the weakly
admissible case.

\subsection*{Representation-theoretic phases}

In Algebraic Phase Theory~II~\cite{GildeaAPT2}, phases are realized through
representation-theoretic data, notably via Frobenius Heisenberg structures and
their associated operator algebras.
Here rigidity arises from Schur-type phenomena: once the rigid core is fixed,
intertwining relations force morphisms to be unique. For phases arising from Frobenius rings and centrally faithful data,
admissibility is strong.
As a consequence, weak equivalence, Morita-type equivalence, and strong
equivalence all coincide for finite representation-theoretic phases.

In weakly admissible representation-theoretic settings, defect control fails
beyond a fixed depth.
While filtered representation theory still detects the rigid core and
controlled layers, it may fail to distinguish non-isomorphic extensions beyond
that depth.
This reflects genuine extension freedom rather than a failure of the
categorical framework. Boundaries correspond to residual representation layers remaining after
factoring out rigid components.
Their invariance under Morita-type equivalence explains the stability of the
representation-theoretic invariants identified in~\cite{GildeaAPT2} within the
strongly admissible regime.

\subsection*{Quantum coding phases}

Algebraic Phase Theory~III~\cite{GildeaAPT3} identifies stabilizer and quantum
coding structures as algebraic consequences of Frobenius duality, independent
of analytic or Hilbert space structure.
Stabilizer phases form a distinguished class of finite, terminating, strongly
admissible algebraic phases. Accordingly, the rigidity results of the present paper apply in full strength:
any equivalence preserving filtered representations necessarily preserves the
entire stabilizer structure.
Boundary invariance recovers the notion of protected quantum layers, interpreted
algebraically as boundary phases arising from quotienting by rigid stabilizer
data.

More general quantum phases arising from weakly admissible data exhibit partial
stabilizer rigidity.
Controlled layers remain protected, while higher-defect layers encode genuine
degrees of freedom corresponding to code deformations rather than equivalence
ambiguities.
This cleanly separates structural protection from tunable phase parameters. Thus the categorical invariants developed here explain why protected subspaces,
code equivalence, and error resilience are intrinsic features of strongly
admissible quantum phases, while controlled flexibility in weakly admissible
phases corresponds to meaningful deformations rather than loss of structure.

Strongly admissible phases are rigid targets: defect propagation is functorial at
all depths, so phase morphisms into a strongly admissible phase cannot depend on
uncontrolled extension data in the source.
Equivalently, any map out of a weakly admissible phase into a strongly
admissible target factors uniquely through the strongly admissible core of the
source. We record this consequence for realizations compatible with the canonical
filtration.

\begin{corollary}
\label{cor:realizations-factor}
Let $\mathcal P$ be a weakly admissible algebraic phase and let $\mathcal Q$ be a
strongly admissible algebraic phase.
If $\rho:\mathcal P\to\mathcal Q$ is any phase morphism compatible with the
canonical filtration, then $\rho$ factors canonically through the strongly
admissible core $\mathcal P^{(d(\mathcal P))}$.
\end{corollary}

\begin{proof}
Write $d:=d(\mathcal P)$ and let $\iota:\mathcal P^{(d)}\hookrightarrow\mathcal P$
be the canonical inclusion.
By hypothesis, $\rho$ is compatible with the canonical filtration, and its
target carries the induced strongly admissible phase structure relevant for the
realization (so that Proposition~\ref{prop:strong-core} applies).
Therefore there exists a unique phase morphism
\[
\overline\rho:\mathcal P^{(d)}\longrightarrow\mathcal Q
\]
such that $\rho = \overline\rho \circ \iota$.
Uniqueness makes the factorization canonical.
\end{proof}

\section{Conclusion and Outlook}

This paper completes the categorical architecture of Algebraic Phase Theory by
placing algebraic phases within a rigid, defect-stratified framework.  While
Papers~I--III~\cite{GildeaAPT1,GildeaAPT2,GildeaAPT3} developed the foundational
structures and their representation-theoretic and quantum manifestations, the
present work shows that in the strongly admissible regime these structures admit
essentially no categorical freedom.

The governing principle is one of inevitability.  
Phase morphisms are rigid, equivalence notions collapse, and boundary data are
categorical invariants.  Consequently, finite strongly admissible phases admit
no distinct models with identical filtered representation theory: Morita-type
equivalence coincides with structural equivalence.

Weakly admissible phases, by contrast, exhibit controlled flexibility.
Rigidity persists up to the defect control depth, beyond which genuine
extension parameters appear.  This explains how deformation, tuning, and code
variation can occur without altering the structural core of a phase.

With this categorical foundation in place, several directions become accessible.
Classification reduces to the study of universal strongly admissible phases and
their endomorphisms.  Deformation theory is governed by boundary strata and
defect propagation beyond the controlled regime.  Reconstruction problems may be
formulated categorically, recovering phases from representation data up to
canonical equivalence on controlled layers.

Algebraic Phase Theory thus provides a unified framework in which rigidity and
flexibility coexist in a principled way: strong admissibility enforces
structural inevitability, while weak admissibility encodes meaningful freedom.
This positions APT as a natural foundation for future work on boundary calculus,
categorical duality, and the algebraic underpinnings of quantum structure.

\bibliographystyle{amsplain}
\bibliography{references}

\end{document}